\documentclass[amstex,12pt, amssymb]{article}

\usepackage{mathtext}
\usepackage[cp1251]{inputenc}
\usepackage[T2A]{fontenc}
\usepackage[dvips]{graphicx}
\usepackage{amsmath}
\usepackage{amssymb}
\usepackage{amsxtra}
\usepackage{latexsym}
\usepackage{ifthen}

\newcounter{lemma}[section]

\newcounter{corollary}[section]

\newcounter{remark}[section]

\newcounter{theorem}[section]

\newcounter{proposition}[section]

\numberwithin{equation}{section}

\begin{document}

\markboth{\centerline{E.~SEVOST'YANOV, S.~SKVORTSOV}}
{\centerline{ON BEHAVIOR OF HOMEOMORPHISMS... }}

\def\cc{\setcounter{equation}{0}
\setcounter{figure}{0}\setcounter{table}{0}}

\overfullrule=0pt


\author{{E.~SEVOST'YANOV, S.~SKVORTSOV}\\}

\title{
{\bf ON BEHAVIOR OF HOMEOMORPHISMS WITH INVERSE MODULUS CONDITIONS}}

\date{\today}
\maketitle

\begin{abstract} We consider some class of homeomorphisms of domains of Euclidean
space, which are more general than quasiconformal mappings. For
these homeomorphisms, we have obtained theorems on local behavior of
it's inverse mappings in a given domain. Under some additional
conditions, we proved results about behavior of mappings mentioned
above in the closure of the domain.
\end{abstract}

\bigskip
{\bf 2010 Mathematics Subject Classification: Primary 30C65;
Secondary 32U20, 31B15}

\section{Introduction}  In Euclidean space, questions connected with the equicontinuity
of quasiconformal mappings and some of their generalizations are
relatively well studied (see., e.g., \cite[Theorem~19.2]{Va},
\cite[Theorem~3.17]{MRV$_2$} and \cite[Lemma~3.12,
Corollary~3.22]{RSY}). The behavior of such classes is also
investigated when this domain is closed (see., e.g.,
\cite[Theorem~3.1]{NP} and \cite[Theorem~3.1]{NP$_1$}). The passage
to inverse mappings in the latter case does not present
difficulties, since, as it is known, the
qua\-si\-con\-for\-ma\-li\-ty of a direct mapping $f$ implies the
quasiconformality of the mapping $f^{\,-1}$ (moreover, the
qua\-si\-con\-for\-ma\-li\-ty constant of the mappings is one and
the same, see. e.g., \cite[Corollary~13.3]{Va}; see. also
\cite[Theorem~34.3]{Va}). In other words, the study of mappings,
inverse to quasiconformal, does not bring anything new in comparison
with investigation of quasiconformal mappings.

The situation essentially changes if instead of quasiconformal
mappings we consider some more general class of homeomorphisms. Let
$M$ means modulus of curve family (see \cite{Va}) and $dm(x)$
corresponds to Lebesque measure in ${\Bbb R}^n.$ Suppose that
mapping $f:D\rightarrow {\Bbb R}^n,$ is defined in domain $D\subset
{\Bbb R}^n,$ $n\geqslant 2,$ and it is satisfying
\begin{equation} \label{eq2*!}
M(f(\Gamma))\leqslant \int\limits_D Q(x)\cdot \rho^n (x)\,
dm(x)\quad \forall\,\,\rho\in {\rm adm\,}\Gamma
\end{equation}
where $Q: D\rightarrow [1,\infty]$ is a certain (given) fixed
function (see, e.g., \cite{MRSY$_4$}). Recall that $\rho\in {\rm
adm}\,\Gamma$ if and only if
%
$$\int\limits_{\gamma}\rho (x)|dx|\geqslant 1\quad
\forall\,\,\gamma\,\in \Gamma\,.$$
%
%
In particular, all conformal and quasiconformal mappings satisfy
(\ref{eq2*!}), where function $Q$ equals 1 or some constant,
respectively (see, e.g., \cite[Theorems~4.6 and 6.10]{MRSY$_1$}).
Note that in case of particular (unbounded) function $Q$ we,
generally speaking, can not replace $f$ by $f^{\,-1}$ in
(\ref{eq2*!}). (For this occasion, see the example~2,
 cited at the end of this work). The study of mappings $g,$
the inverses of which satisfy the relation (\ref{eq2*!}) is a
separate topic for research. In this note we are interested in the
local behavior of such mappings $g$ in the domain
$D^{\,\prime}=f(D),$ $f=g^{\,-1},$ and also in
$\overline{D^{\,\prime}}.$

It is necessary to take into the early results of the first
author~\cite{Sev$_3$}, where mappings $g$ with similar conditions
were also studied. The main result is contained in
\cite[Theorem~6.1]{Sev$_3$} and it is proved under the condition
that two points of the domain are fixed by mappings, that it is
difficult to call an optimal constraint. In particular, among linear
fractional automorphisms of the unit circle onto itself is at most
one such mapping, in view of which the indicated condition turns out
to be meaningless. Our main goal is to study analogous families of
mappings with a rejection of any conditions normalization. As
example~1 shows at the end of the paper, it essentially enriches the
results obtained in the article from the point of view of
applications.

\medskip
Main definition and denotes used below can be found in monographs
\cite{Va} and \cite{MRSY} and therefore omitted. Let $E,$ $F\subset
\overline{{\Bbb R}^n}$ are arbitrary sets. Further $\Gamma(E,F,D)$
we denote the family of all path
$\gamma:[a,b]\rightarrow\overline{{\Bbb R}^n},$ that connect $E$ and
$F$ in $D,$ i.e $\gamma(a)\in E,\,\gamma(b)\in F$ и $\gamma(t)\in D$
for $t\in(a,\,b).$ Recall that the domain $D\subset {\Bbb R}^n$ is
called {\it locally connected at the point} $x_0\in\partial D,$ if
for every neighborhood $U$ of a point $x_0$ there is a neighborhood
$V\subset U$ of a point $x_0$ such that $V\cap D$ is connected. The
domain $D$ is locally connected in the $\partial D,$ if $D$ is
locally connected at every point $x_0\in\partial D.$ The boundary of
$D$ is called {\it weakly flat} at a point $x_0\in
\partial D,$ if for every $P>0$ and every neighborhood $U$
of the point $x_0,$ there is a neighborhood $V\subset U$ of $x_0$
such that $M(\Gamma(E, F, D))>P$ for all continua $E, F\subset D,$
intersecting $\partial U$ and $\partial V.$ The boundary of the
domain $D$ is weakly flat, if it is weakly flat at every point of
boundary of $D.$

\medskip For domains $D, D^{\,\prime}\subset {\Bbb R}^n,$ $n\geqslant 2,$
and arbitrary Lebesgue measurable function $Q: {\Bbb R}^n\rightarrow
[1, \infty],$ $Q(x)\equiv 0$ for $x\not\in D,$ denote ${\frak
R}_Q(D, D^{\,\prime})$ the family of all mappings
$g:D^{\,\prime}\rightarrow D$ such that $f=g^{\,-1}$ is
homeomorphism of the domain $D$ onto $D^{\,\prime}$
satisfying(\ref{eq2*!}). The following assertion is valid.

\begin{theorem}\label{th1}
{\sl Suppose that $\overline{D}$ and $\overline{D^{\,\prime}}$ are a
compacts in ${\Bbb R}^n.$  If $Q\in L^1(D),$ then the family ${\frak
R}_Q(D, D^{\,\prime})$ is equicontinuous in $D^{\,\prime}.$ }
\end{theorem}

For the number $\delta>0,$ domains $D$ and $D^{\,\prime}\subset
{\Bbb R}^n,$ $n\geqslant 2,$ continuum $A\subset D$ and arbitrary
Lebesgue measurable function $Q: {\Bbb R}^n\rightarrow [1, \infty],$
$Q(x)\equiv 0$ for $x\not\in D,$ denote by ${\frak S}_{\delta, A, Q
}(D, D^{\,\prime})$ the family of all mappings
$g:D^{\,\prime}\rightarrow D$ such that $f=g^{\,-1}$ is a
homeomorphism of the domain $D$ onto $D^{\,\prime}$ satisfying
(\ref{eq2*!}), wherein ${\rm diam}\,f(A)\geqslant\delta.$ The
following assertion is valid.

\begin{theorem}\label{th2}
{\sl Suppose that the domain $D$ is locally connected at all
boundary points, $\overline{D}$ and $\overline{D^{\,\prime}}$ are
compacts in ${\Bbb R}^n,$ and the domain $D^{\,\prime}$ has a weakly
flat boundary. We also suppose that any path-connected component
$\partial D^{\,\prime}$ is non-degenerate continuum. If $Q\in
L^1(D),$ then each mapping $g\in {\frak S}_{\delta, A, Q }(D,
D^{\,\prime})$ extends by continuity to the mapping
$\overline{g}:\overline{D^{\,\prime}}\rightarrow \overline{D},$
$\overline{g}|_{D^{\,\prime}}=g,$ in addition,
$\overline{g}(\overline{D^{\,\prime}})=\overline{D}$ and family
${\frak S}_{\delta, A, Q }(\overline{D}, \overline{D^{\,\prime}}),$
consisting of all extended mappings
$\overline{g}:\overline{D^{\,\prime}}\rightarrow \overline{D},$ is
equicontinuous in $\overline{D^{\,\prime}}.$ }
\end{theorem}

\section{Auxiliary information}
First of all, we establish two elementary statements that play an
important role in the proof of the main results. Let $I$ be an open,
closed or half-open interval in ${\Bbb R}.$ As usual, for a curve
$\gamma: I\rightarrow {\Bbb R}^n$ suppose:
$$|\gamma|=\{x\in {\Bbb R}^n: \exists\, t\in [a, b]:
\gamma(t)=x\}\,,$$
wherein, $|\gamma|$ is called {\it carrier (image)  of the curve}
$\gamma.$ We say that the curve $\gamma$ lies in the domain $D,$ if
$|\gamma|\subset D,$ in addition, we will say that the curves
$\gamma_1$ and $\gamma_2$ do not intersect if their carriers do not
intersect. The curve $\gamma:I\rightarrow {\Bbb R}^n$ is called {\it
Jordan arc}, if $\gamma$ is a homeomorphism on $I.$ The following
(almost obvious) assertion is valid.

\begin{lemma}\label{lem1}{\sl\,
Let $D$ be a domain in ${\Bbb R}^n,$ $n\geqslant 2,$ locally
connected on its boundary.  Then any two pairs of different points
$a\in D, b\in \overline{D},$ и $c\in D, d\in \overline{D}$ can be
joined by disjoint curves $\gamma_1:[0, 1]\rightarrow \overline{D}$
and $\gamma_2:[0, 1]\rightarrow \overline{D},$  so, that
$\gamma_i(t)\in D$ for all $t\in (0, 1),$ $i=1,2,$ $\gamma_1(0)=a,$
$\gamma_1(1)=b,$ $\gamma_2(0)=c,$ $\gamma_2(1)=d.$}
\end{lemma}

\begin{proof}
Notice, that the points of the domain are locally connected on the
boundary and are accessible from within the domain by means of
curves (see, e.g., \cite[Proposition~13.2]{MRSY}). In this case, if
$n\geqslant 3,$ we connect the points $a$ and $b$ by an arbitrary
Jordan arc $\gamma_1$ in the domain $D,$ not passing through the
points $c$ and $d$ (which is possible in view of the local
connection of $D$ on the boundary and the transition from the curve
to the broken line if it is necessary). Then $\gamma_1$ does not
divide the domain $D$ as a set of topological dimension 1 (see
\cite[Corollary~1.5.IV]{HW}), which ensures the existence of the
desired curve $\gamma_2.$ Thus, in the case of $n\geqslant 3$ the
assertion of Lemma~\ref{lem1} is established.

Now let $n=2,$ then again the points $c$ and $d$ does not divide the
domain $D$ (\cite[Corollary~1.5.IV]{HW}). In this case, you can also
connect points $a$ and $b$ by a Jordan arc $\gamma_1$ in $D,$ that
does not pass through the points $c$ and $d.$ In view of the Antoine
theorem (see \cite[Theorem~4.3, \S\,4]{Keld}) the domain $D$ can be
mapped onto some domain $D^{\,*}$  by means of a flat homeomorphism
$\varphi:{\Bbb R}^2\rightarrow {\Bbb R}^2$ so, that
$\varphi(\gamma_1)=J$ и $J$ is a segment in $D^{\,*}.$ We also note
that the boundary points of the domain $D^{\,*}$ are reachable from
within $D^{\,*}$ by means of curves. In this way, we can connect
points $\varphi(c)$ and $\varphi(d)$ in $D^{\,*}$ by a Jordan arc
$\alpha_2:[0, 1]\rightarrow \overline{D^{\,*}},$ which lies entirely
in $D^{\,*},$ except perhaps its end point $\alpha_2(1)=\varphi(d).$

It remains to show that the curve $\alpha_2$ can be chosen so that
it does not intersect the segment $J.$ In fact, let $\alpha_2$
crosses $J,$ and let $t_1$ and $t_2$ are, respectively, the largest
and the smallest values $t\in [0, 1],$ for which $\alpha_2(t)\in
|J|.$ Suppose also that
$$J=J(s)=\varphi(a)+ (\varphi(b)-\varphi(a))s, \quad s\in [0, 1]$$
is a parametrization of the interval $J.$ Let $\widetilde{s_1}$ and
$\widetilde{s_2}\in (0, 1)$ be such that
$J(\widetilde{s_1})=\alpha_2(t_1)$ and
$J(\widetilde{s_2})=\alpha_2(t_2).$ Suppose
$s_2=\max\{\widetilde{s_1}, \widetilde{s_2}\}.$ Let
$e_1=\varphi(b)-\varphi(a)$ and $e_2$ is a unit vector, orthogonal
to $e_1,$ then the set
$$P_{\varepsilon}=\{x=\varphi(a)+
x_1e_1+x_2e_2,\quad x_1\in (-\varepsilon, s_2+\varepsilon), \quad
x_2\in (-\varepsilon, \varepsilon)\}\,,\quad \varepsilon>0\,,$$
is a rectangle containing $|J_1|,$ where $J_1$ is a restriction of
$J$ to a segment $[0, s_2]$ (see picture~\ref{fig1}).
\begin{figure}[h]
\centerline{\includegraphics[scale=0.5]{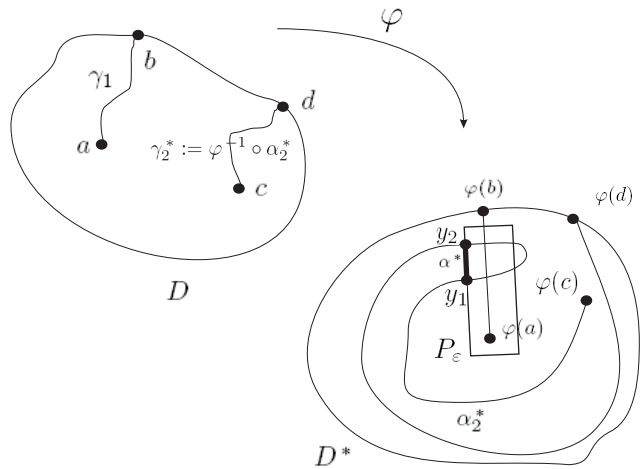}} \caption{The
possibility of connecting two pairs of points by curves in the
domain}\label{fig1}
\end{figure}
We choose that $\varepsilon>0$ so that $\varphi(c)\not\in
P_{\varepsilon},$ ${\rm dist}\,(P_{\varepsilon}, \partial
D^{\,*})>\varepsilon.$ In view of \cite[Theorem~1.I, ch.~5, \S\,
46]{Ku}) the curve $\alpha_2$ crosses $\partial P_{\varepsilon}$ for
some $T_1<t_1$ and $T_2>t_2.$ Let $\alpha_2(T_1)=y_1$ and
$\alpha_2(T_2)=y_2.$ Since $\partial P_{\varepsilon}$ is a connected
set, it is possible to connect points $y_1$ and $y_2$ of the curve
$\alpha^{\,*}(t):[T_1, T_2]\rightarrow
\partial P_{\varepsilon}.$ Finally, we put
$$\alpha_2^{\,*}(t)\quad =\quad\left\{
\begin{array}{rr}
\alpha_2(t), & t\in [0, 1]\setminus [T_1, T_2],\\
\alpha^{\,*}(t), & t\in [T_1, T_2]\end{array} \right.$$
and $\gamma^{\,*}_2:=\varphi^{\,-1}\circ \alpha_2^{\,*}.$ Then
$\gamma_1$ connects $a$ and $b$ in $D,$ and $\gamma_2^{\,*}$
connects $c$ and $d$ in $D,$ while $\gamma_1$ and $\gamma_2^{\,*}$
do not intersect, which should be established.~$\Box$
\end{proof}

\medskip
Above we introduced the concept of a weak plane of the boundary of
the region, without mentioning, at the same time, internal points.
The following lemma contains the assertion that at the indicated
points the property of the <<weak plane>> always takes place.

\begin{lemma}\label{lem2}
{\sl\, Let $D$ be a domain in ${\Bbb R}^n,$ $n\geqslant 2,$ and
$x_0\in D.$ Then for every $P>0$ and for for any neighborhood $U$ of
the point $x_0$ there is a neighborhood $V\subset U$ of the same
point such that $M(\Gamma(E, F, D))>P$ for arbitrary continua $E,
F\subset D,$ intersecting $\partial U$ and $\partial V.$}
\end{lemma}

\begin{proof}
Let $U$ be an arbitrary neighborhood of $x_0.$ Let`s choose
$\varepsilon_0>0$ so that $\overline{B(x_0, \varepsilon_0)}\subset
D\cap U.$ Let $c_n$ be a positive constant, defined in the
relation~(10.11) in \cite{Va}, and the number $\varepsilon\in(0,
\varepsilon_0)$ is so small that
$c_n\cdot\log\frac{\varepsilon_0}{\varepsilon}>P.$ Suppose
$V:=B(x_0, \varepsilon).$ Let $E, F$ be arbitrary continua
intersecting $\partial U$ and $\partial V,$ then also $E$ and $F$
intersecting $S(x_0, \varepsilon_0)$ and $\partial V$ (see
\cite[Theorem~1.I, ch.~5, \S\, 46]{Ku}). The necessary conclusion
follows on the basis of~\cite[par.~10.12]{Va}, because the
$$M(\Gamma(E, F, D))\geqslant c_n\cdot\log\frac{\varepsilon_0}{\varepsilon}>P\,.\quad \Box$$
\end{proof}

\section{Proof of Theorem~\ref{th1}} We prove the theorem~\ref{th1} by contradiction. Suppose, the family ${\frak
R}_Q(D, D^{\,\prime})$ is not equ\-i\-con\-ti\-nuous at some point
$y_0\in D^{\,\prime},$ in other words, there are $y_0\in
D^{\,\prime}$ and $\varepsilon_0>0,$ such that for any $m\in {\Bbb
N}$ there exists an element $y_m\in D^{\,\prime},$ $|y_m-y_0|<1/m,$
and a homeomorphism $g_m\in{\frak R}_Q(D, D^{\,\prime}),$ for which
\begin{equation}\label{eq13***}
|g_m(y_m)-g_m(y_0)|\geqslant \varepsilon_0\,.
\end{equation}
We draw a line $r=r_m(t)=g_m(y_0)+(g_m(y_m)-g_m(y_0))t,$
$-\infty<t<\infty$ through $g_m(y_m)$ and $g_m(y_0)$
(see~picture~\ref{fig2}).
\begin{figure}[h]
\centerline{\includegraphics[scale=0.6]{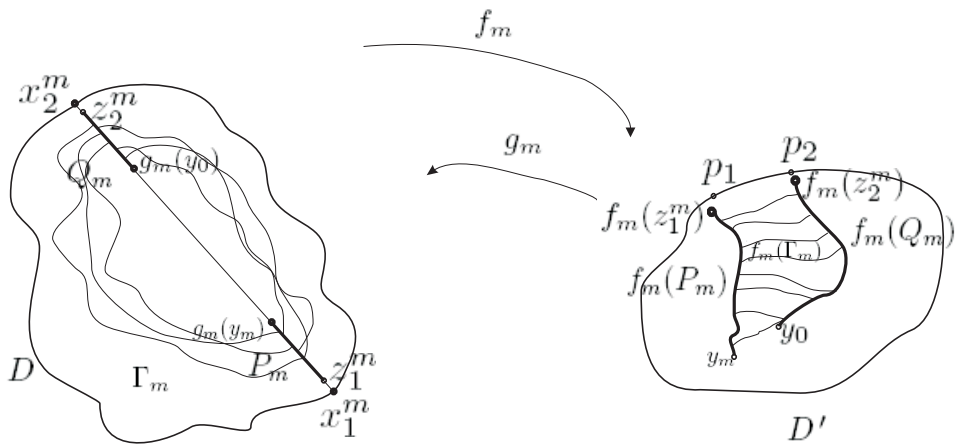}} \caption{To
the proof of the theorem~\ref{th1}}\label{fig2}
\end{figure}
Note that this line $r=r_m(t)$ for $t\geqslant 1$ must intersect the
domain $D$ in view of \cite[Theorem~1.I, ch.~5, \S\, 46]{Ku}), since
the domain $D$ is bounded; thus, there exists $t_1^m\geqslant 1$
such that $r_m(t^m_1)=x^m_1\in
\partial D.$ Without loss of generality we can assume that $r_m(t)\in
D$ for all $t\in [1, t^m_1),$  then the segment
$\gamma^m_1(t)=g_m(y_0)+(g_m(y_m)-g_m(y_0))t,$ $t\in [1, t^m_1],$
belongs to $D$ for all $t\in [1, t^m_1),$
$\gamma^m_1(t^m_1)=x^m_1\in
\partial D$ and $\gamma^m_1(1)=g_m(y_m).$
In view of analogous considerations, there are $t^m_2<0$ and a
segment $\gamma^m_2(t)=g_m(y_0)+(g_m(y_m)-g_m(y_0))t,$ $t\in [t^m_2,
0],$ such that $\gamma^m_2(t^m_2)=x^m_2\in
\partial D,$ $\gamma^m_2(0)=g_m(y_0)$ и $\gamma^m_2(t)$
belongs to $D$ for all $t\in (t^m_2, 0].$ Put $f_m:=g_m^{\,-1}.$
Since $f_m$ is a homeomorphism, for each fixed $m\in {\Bbb N}$ the
limit sets $C(f_m, x_1^m)$ and $C(f_m, x_2^m)$ of mappings $f_m$ at
the corresponding boundary points $x_1^m, x_2^m\in \partial D$ lie
on $\partial D^{\,\prime}$ (see \cite[Proposition~13.5]{MRSY}).
Consequently, there is a point $z^m_1\in D\cap |\gamma^m_1|$ such
that ${\rm dist}\,(f_m(z_1^m),
\partial D^{\,\prime})<1/m.$ As $\overline{D^{\,\prime}}$ is compact, it can be assumed that the sequence
$f_m(z_1^m)\rightarrow p_1\in \partial D^{\,\prime}$ for
$m\rightarrow\infty.$ Similarly, there is a sequence $z^m_2\in D\cap
|\gamma^m_2|$ such that ${\rm dist}\,(f_m(z^m_2),
\partial D^{\,\prime})<1/m$ and
$f_m(z^m_2)\rightarrow p_2\in
\partial D^{\,\prime}$ for $m\rightarrow\infty.$

Let $P_m$ be the part of the interval $\gamma_1^m,$ enclosed between
the points $g_m(y_m)$ and $z^m_1,$ and $Q_m$ be the part of the
interval $\gamma_2^m,$ enclosed between the points $g_m(y_0)$ and
$z^m_2.$
By construction and by (\ref{eq13***}), ${\rm dist}\,(P_m,
Q_m)\geqslant\varepsilon_0>0.$ Let $\Gamma_m=\Gamma(P_m, Q_m, D),$
then the function
$$\rho(x)= \left\{
\begin{array}{rr}
\frac{1}{\varepsilon_0}, & x\in D,\\
0,  &  x\notin  D
\end{array}
\right. $$
is admissible for the family $\Gamma_m,$ since for an arbitrary
(locally rectifiable) curve $\gamma\in \Gamma_m$ it is completed
$\int\limits_{\gamma}\rho(x)|dx|\geqslant
\frac{l(\gamma)}{\varepsilon_0}\geqslant 1$ (where $l(\gamma)$
denotes the length of the curve $\gamma$). Since by the hypothesis
the mappings $f_m$ satisfy (\ref{eq2*!}) we obtain:
\begin{equation}\label{eq14***}
M(f_m(\Gamma_m))\leqslant \frac{1}{\varepsilon_0^n}\int\limits_{D}
Q(x)\,dm(x):=c<\infty\,,
\end{equation}
as $Q\in L^1(D).$
On the other hand, ${\rm diam}\,f_m(P_m)\geqslant |y_m-f_m(z^m_1)|
\geqslant (1/2)\cdot|y_0-p_1|>0$ and ${\rm diam}\,f_m(Q_m)\geqslant
|y_0-f_m(z^m_2)| \geqslant (1/2)\cdot|y_0-p_2|>0$ for large $m\in
{\Bbb N},$ in addition
$${\rm dist}\,(f_m(P_m), f_m(Q_m))\leqslant |y_m-y_0|\rightarrow 0,\quad m\rightarrow
\infty\,.$$ Then, in view of Lemma~\ref{lem2}
$$M(f_m(\Gamma_m))=M(f_m(P_m), f_m(Q_m), D^{\,\prime})\rightarrow\infty\,,\quad m\rightarrow\infty\,,$$
which contradicts relation (\ref{eq14***}). This contradiction
indicates that the assumption in (\ref{eq13***}) is erroneous, which
completes the proof of the theorem.~$\Box$

\section{On the behavior of mappings in the closure of a domain}

Let us pass to the question of the global behavior of mappings. The
following assertion indicates that for sufficiently good domains and
mappings with condition~(\ref{eq2*!}) the image of a fixed continuum
under these mappings can not approach the boundary of the
corresponding domain as soon as the Euclidean of the diameter of
this continuum is bounded from below (see also~\cite[Theorems~21.13
and 21.14]{Va}).

\begin{lemma}\label{lem3}
{\sl\, Suppose that the domain $D$ is locally path-connected on
$\overline{D},$ $\overline{D}$ and $\overline{D^{\,\prime}}$ are
compact sets in ${\Bbb R}^n,$ $n\geqslant 2,$ $D^{\,\prime}$ has a
weakly flat boundary, $Q\in L^1(D)$ and there is no connected
component of the boundary $\partial D^{\,\prime}$ degenerating to a
point. Let $f_m:D\rightarrow D^{\,\prime}$ be a sequence of
homeomorphisms of the domain $D$ onto the domain $D^{\,\prime}$ with
the condition~(\ref{eq2*!}). Let there also be a continuum $A\subset
D$ and a number $\delta>0$ such that ${\rm diam\,} f_m(A)\geqslant
\delta>0$ for all $m=1,2,\ldots .$ Then there exists $\delta_1>0$
such that
$${\rm dist\,}(f_m(A),
\partial D^{\,\prime})>\delta_1>0\quad \forall\,\, m\in {\Bbb
N}\,.$$}
\end{lemma}

\begin{proof}
Suppose, the contrary situation, that for each $k\in {\Bbb N}$ there
exists $m=m_k:$ ${\rm dist\,}(f_{m_k}(A),
\partial D^{\,\prime})<1/k.$ Without loss of generality we can assume that the sequence $m_k$ is monotonically increasing. By condition $\overline{D^{\,\prime}}$ is compact, therefore $\partial
D^{\,\prime}$ is also compact as a closed subset of the compactum
$\overline{D^{\,\prime}}.$ In addition, $f_{m_k}(A)$ is compact as a
continuous image of the compactum $A$ under the mapping $f_{m_k}.$
Then there are $x_k\in f_{m_k}(A)$ and $y_k\in
\partial D^{\,\prime}$ such that ${\rm dist\,}(f_{m_k}(A),
\partial D^{\,\prime})=|x_k-y_k|<1/k$ (see~picture~\ref{fig3}).
\begin{figure}[h]
\centerline{\includegraphics[scale=0.6]{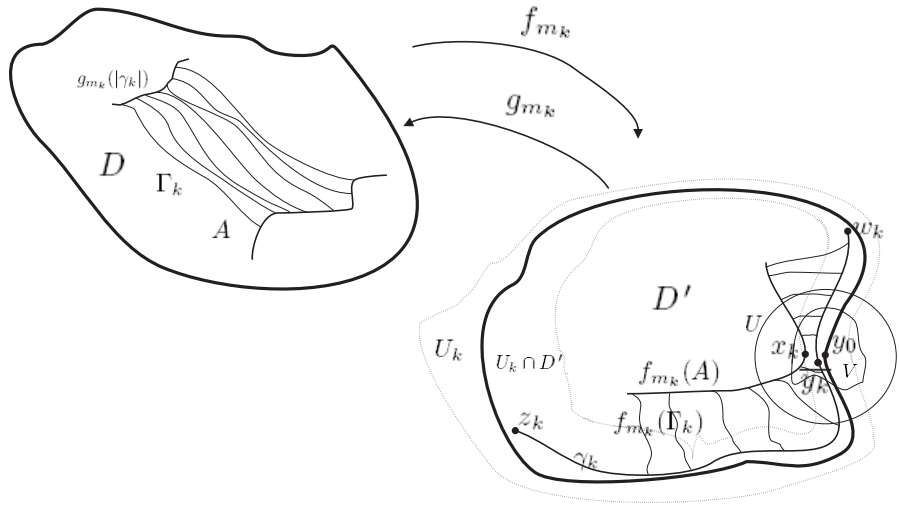}} \caption{To
the proof of Lemma~\ref{lem3}}\label{fig3}
\end{figure}
As $\partial D^{\,\prime}$ since $A$ is compact, we can assume that
$y_k\rightarrow y_0\in \partial D^{\,\prime},$ $k\rightarrow
\infty;$ then also
%
$$x_k\rightarrow y_0\in \partial D^{\,\prime},\quad k\rightarrow
\infty\,.$$
%
Let $K_0$ be a connected component $\partial D^{\,\prime},$
containing the point $y_0,$ then, obviously $K_0$ is a continuum in
${\Bbb R}^n.$ Since $D^{\,\prime}$ has a weakly flat boundary, for
each $k\in {\Bbb N}$ the mapping $g_{m_k}:=f_{m_k}^{\,-1}$ extends
to a continuous mapping
$\overline{g}_{m_k}:\overline{D^{\,\prime}}\rightarrow \overline{D}$
(see~\cite[Theorem~4.6]{MRSY}), furthermore , $\overline{g}_{m_k}$
uniformly continuous on $\overline{D^{\,\prime}}$ as a mapping that
is continuous on a compactum. Then for every $\varepsilon>0$ there
is $\delta_k=\delta_k(\varepsilon)<1/k$ such that
\begin{equation}\label{eq3}
|\overline{g}_{m_k}(x)-\overline{g}_{m_k}(x_0)|<\varepsilon \quad
\forall\,\, x,x_0\in \overline{D^{\,\prime}},\quad
|x-x_0|<\delta_k\,, \quad \delta_k<1/k\,.
\end{equation}
Let $\varepsilon>0$ be an arbitrary number with the condition
\begin{equation}\label{eq5}
\varepsilon<(1/2)\cdot {\rm  dist}\,(\partial D, A)\,,
\end{equation}
where $A$ is a continuum from the condition of the lemma. For each
fixed $k\in {\Bbb N}$ we consider the set
$$B_k:=\bigcup\limits_{x_0\in K_0}B(x_0, \delta_k)\,,\quad k\in {\Bbb N}\,.$$
Note that $B_k$ is an open set containing $K_0,$ in other words,
$B_k$ is a neighborhood of the continuum $K_0.$ In view of
\cite[Lemma~2.2]{HK} there exists a neighborhood $U_k\subset B_k$ of
the continuum $K_0,$ such that $U_k\cap D^{\,\prime}$ is connected.
Without loss of generality, we can assume that $U_k$ is an open set,
then $U_k\cap D^{\,\prime}$ is also linearly connected
(see~\cite[Proposition~13.1]{MRSY}). Let ${\rm diam}\,K_0=m_0,$ then
there exist $z_0, w_0\in K_0$ such that ${\rm
diam}\,K_0=|z_0-w_0|=m_0.$ Hence, we can choose sequences
$\overline{y_k}\in U_k\cap D^{\,\prime},$ $z_k\in U_k\cap
D^{\,\prime}$ and $w_k\in U_k\cap D^{\,\prime}$ such that
$z_k\rightarrow z_0,$ $\overline{y_k}\rightarrow y_0$ and
$w_k\rightarrow w_0$ for $k\rightarrow\infty.$ We can assume that
\begin{equation}\label{eq2}
|z_k-w_k|>m_0/2,\quad \forall\,\, k\in {\Bbb N}\,.
\end{equation}
We connect consecutively the points $z_k,$ $\overline{y_k}$ and
$w_k$ of the curve $\gamma_k$ in $U_k\cap D^{\,\prime}$ (this is
possible, since $U_k\cap D^{\,\prime}$ is path-connected). Let
$|\gamma_k|$ be, as usual, the carrier (image) of the curve
$\gamma_k$ in $D^{\,\prime}.$ Then $g_{m_k}(|\gamma_k|)$ is a
compact set in $D.$ Let $x\in|\gamma_k|,$ then there is $x_0\in
K_0:$ $x\in B(x_0, \delta_k).$ We will fix $\omega\in A\subset D.$
Because the $x\in|\gamma_k|,$ then $x$ is an interior point of the
domain $D^{\,\prime},$ so we have the right to write $g_{m_k}(x)$
instead of $\overline{g}_{m_k}(x)$ for the indicated $x.$ In this
case, from (\ref{eq3}) and (\ref{eq5}), in view of the triangle
inequality, for large $k\in {\Bbb N}$ we obtain:
$$|g_{m_k}(x)-\omega|\geqslant
|\omega-\overline{g}_{m_k}(x_0)|-|\overline{g}_{m_k}(x_0)-g_{m_k}(x)|\geqslant$$
\begin{equation}\label{eq4}
\geqslant {\rm  dist}\,(\partial D, A)-(1/2)\cdot{\rm
dist}\,(\partial D, A)=(1/2)\cdot{\rm  dist}\,(\partial D,
A)>\varepsilon\,.
\end{equation}
Passing to (\ref{eq4}) to $\ inf,$ over all $x\in |\gamma_k|$ and
all $\omega\in A,$ we obtain:
\begin{equation}\label{eq6}
{\rm dist}\,(g_{m_k}(|\gamma_k|), A)>\varepsilon, \quad\forall\,\,
k=1,2,\ldots \,.
\end{equation}
In view of (\ref{eq6}) the length of an arbitrary curve joining
compacta $g_{m_k}(|\gamma_k|)$ and $A$ in $D,$ not less than
$\varepsilon.$ Put $\Gamma_k:=\Gamma(g_{m_k}(|\gamma_k|), A, D),$
then the function $\rho(x)=1/\varepsilon$ for $x\in D$ and
$\rho(x)=0$ for $x\not\in D$ is admissible for $\Gamma_k,$ since
$\int\limits_{\gamma}\rho(x)|dx|\geqslant
\frac{l(\gamma)}{\varepsilon}\geqslant 1$ for $\gamma\in\Gamma_k$
(where $l(\gamma)$ denotes the length of the curve $\gamma$). By the
definition of mappings $f_{m_k}$ in (\ref{eq2*!}), we have:
\begin{equation}\label{eq4B}
M(f_{m_k}(\Gamma_k))\leqslant
\frac{1}{\varepsilon^n}\int\limits_DQ(x)\,dm(x)=c=c(\varepsilon,
Q)<\infty\,,
\end{equation}
Since by hypothesis $Q\in L^1(D).$

\medskip
We now show that we arrive at a contradiction with (\ref{eq4B}) in
view of the weak boundary plane $\partial D^{\,\prime}.$ We choose
at the point $y_0\in \partial D^{\,\prime}$ the ball $U:=B(y_0,
r_0),$ where $r_0>0$ and $r_0<\min\{\delta/4, m_0/4\},$ $\delta$ --
is a number from the condition of the lemma and ${\rm
diam}\,K_0=m_0.$ Notice, that $|\gamma_k|\cap U\ne\varnothing\ne
|\gamma_k|\cap (D^{\,\prime}\setminus U)$ for sufficiently large
$k\in{\Bbb N},$ because the ${\rm diam\,} |\gamma_k|\geqslant
m_0/2>m_0/4$ и $\overline{y_k}\in |\gamma_k|,$
$\overline{y_k}\rightarrow y_0$ for $k\rightarrow\infty.$ In view of
the same considerations $f_{m_k}(A)\cap U\ne\varnothing\ne
f_{m_k}(A)\cap (D^{\,\prime}\setminus U).$ As $|\gamma_k|$ and
$f_{m_k}(A)$ are continua, then
\begin{equation}\label{eq8}
f_{m_k}(A)\cap \partial U\ne\varnothing, \quad|\gamma_k|\cap
\partial U\ne\varnothing\,,
\end{equation}
see~\cite[Theorem~1.I, гл.~5, \S\, 46]{Ku}. For a fixed $P>0,$ let
$V\subset U$ is a neighborhood of the point $y_0,$ corresponding to
the definition of a weakly flat boundary, that is, such that for any
continua $E, F\subset D^{\,\prime}$ with condition $E\cap
\partial U\ne\varnothing\ne E\cap \partial V$ and $F\cap \partial
U\ne\varnothing\ne F\cap \partial V$ is satisfied the inequality
\begin{equation}\label{eq9}
M(\Gamma(E, F, D^{\,\prime}))>P\,.
\end{equation}
We note that for sufficiently large $k\in {\Bbb N}$
\begin{equation}\label{eq10}
f_{m_k}(A)\cap \partial V\ne\varnothing, \quad|\gamma_k|\cap
\partial V\ne\varnothing\,.
\end{equation}
Indeed $\overline{y_k}\in |\gamma_k|,$ $x_k\in f_{m_k}(A),$ where
$x_k, \overline{y_k}\rightarrow y_0\in V$ for $k\rightarrow\infty,$
therefore $|\gamma_k|\cap V\ne\varnothing\ne f_{m_k}(A)\cap V$ for
large $k\in {\Bbb N}.$ Besides ${\rm diam}\,V\leqslant {\rm
diam}\,U=2r_0<m_0/2$ and, since, ${\rm diam }|\gamma_k|>m_0/2$ in
view of (\ref{eq2}), then $|\gamma_k|\cap (D^{\,\prime}\setminus
V)\ne\varnothing.$ Then $|\gamma_k|\cap\partial V\ne\varnothing$
(see~\cite[Theorem~1.I, ch.~5, \S\, 46]{Ku}). Similarly, ${\rm
diam}\,V\leqslant {\rm diam}\,U=2r_0<\delta/2$ and, since ${\rm
diam}\,f_{m_k}(A)>\delta$ by hypothesis, then $f_{m_k}(A)\cap
(D^{\,\prime}\setminus V)\ne\varnothing.$ In view
of~\cite[Theorem~1.I, ch.~5, \S\, 46]{Ku} we have: $f_{m_k}(A)\cap
\partial V\ne\varnothing.$ The relations in (\ref{eq10})
are established.

\medskip
Thus, according to (\ref{eq8}), (\ref{eq9}) and (\ref{eq10}), we get
that
\begin{equation}\label{eq11}
M(\Gamma(f_{m_k}(A), |\gamma_k|, D^{\,\prime}))>P\,.
\end{equation}
Notice, that $\Gamma(f_{m_k}(A), |\gamma_k|,
D^{\,\prime})=f_{m_k}(\Gamma(A, g_{m_k}(|\gamma_k|),
D))=f_{m_k}(\Gamma_k),$ so that inequality (\ref{eq11}) can be
rewritten in the form
$$M(\Gamma(f_{m_k}(A), g_{m_k}(|\gamma_k|), D))=M(f_{m_k}(\Gamma_k))>P\,,$$
which contradicts inequality~(\ref{eq4B}). The resulting
contradiction indicates the incorrectness of the original assumption
${\rm dist\,}(f_{m_k}(A),
\partial D^{\,\prime})<1/k.$ The lemma is proved.~$\Box$
\end{proof}

\medskip
{\it Proof of Theorem~\ref{th2}}. Since $D^{\,\prime}$ has a weakly
flat boundary, each $g\in {\frak S}_{\delta, A, Q }(D,
D^{\,\prime})$ extends to a continuous mapping
$\overline{g}:\overline{D^{\,\prime}}\rightarrow \overline{D}$
(see~\cite[Theorem~4.6]{MRSY}).

We verify equality
$\overline{g}(\overline{D^{\,\prime}})=\overline{D}.$ In fact, by
definition
$\overline{g}(\overline{D^{\,\prime}})\subset\overline{D}.$ It
remains to show the converse inclusion $\overline{D}\subset
\overline{g}(\overline{D^{\,\prime}}).$ Let $x_0\in \overline{D},$
then we show that $x_0\in \overline{g}(\overline{D^{\,\prime}}).$ If
$x_0\in \overline{D},$ then either $x_0\in D,$ or $x_0\in
\partial D.$ If $x_0\in D,$ then there is nothing to prove, since by hypothesis
$\overline{g}(D^{\,\prime})=D.$  Now let $x_0\in \partial D,$ then
there be $x_k\in D$ and $y_k\in D^{\,\prime}$ such that
$x_k=\overline{g}(y_k)$ and $x_k\rightarrow x_0$ for
$k\rightarrow\infty.$ Since $\overline{D^{\,\prime}}$ is compact, we
can assume that $y_k\rightarrow y_0\in \overline{D^{\,\prime}}$ for
$k\rightarrow\infty.$ Since $f=g^{\,-1}$ is a homeomorphism, then
$y_0\in
\partial D^{\,\prime}.$ Since $\overline{g}^{\,-1}$ is continuous in $\overline{D^{\,\prime}},$
$\overline{g}(y_k)\rightarrow \overline{g}(y_0).$ However, in this
case, $\overline{g}(y_0)=x_0,$ since $\overline{g}(y_k)=x_k$ and
$x_k\rightarrow x_0,$ $k\rightarrow\infty.$ Hence, $x_0\in
\overline{g}(\overline{D^{\,\prime}}).$ The inclusion
$\overline{D}\subset \overline{g}(\overline{D^{\,\prime}})$ is
proved and, hence,
$\overline{D}=\overline{g}(\overline{D^{\,\prime}}),$ as required.

Equicontinuity of curve family ${\frak S}_{\delta, A, Q
}(\overline{D}, \overline{D^{\,\prime}})$ at interior points
$D^{\,\prime}$ is the result of the theorem~\ref{th1}. It remains to
show that this family is equicontinuous at the boundary points. We
carry out the proof by contradiction. Suppose we find a point
$z_0\in \partial D^{\,\prime},$ a number $\varepsilon_0>0$ and
sequences $z_m\in \overline{D^{\,\prime}},$ $z_m\rightarrow z_0$ for
$m\rightarrow\infty$ and $\overline{g}_m\in {\frak S}_{\delta, A, Q
}(\overline{D}, \overline{D^{\,\prime}})$ such that
\begin{equation}\label{eq12}
|\overline{g}_m(z_m)-\overline{g}_m(z_0)|\geqslant\varepsilon_0,\quad
m=1,2,\ldots .
\end{equation}
Put $g_m:=\overline{g}_m|_{D^{\,\prime}}.$ Since $g_m$ extends by
continuity to the boundary of $D^{\,\prime},$ we can assume that
$z_m\in D^{\,\prime}$ and, hence, $\overline{g}_m(z_m)=g_m(z_m).$ In
addition, there is one more sequence $z^{\,\prime}_m\in
D^{\,\prime},$ $z^{\,\prime}_m\rightarrow z_0$ for
$m\rightarrow\infty,$ such that
$|g_m(z^{\,\prime}_m)-\overline{g}_m(z_0)|\rightarrow 0$ for
$m\rightarrow\infty.$
%
Since $\overline{D}$ is compact, we can assume that the sequences
$g_m(z_m)$ and $\overline{g}_m(z_0)$ are convergent for
$m\rightarrow\infty.$ Let $g_m(z_m)\rightarrow \overline{x_1}$ and
$\overline{g}_m(z_0)\rightarrow \overline{x_2}$ for
$m\rightarrow\infty.$ By continuity of the modulus from (\ref{eq12})
it follows that $\overline{x_1}\ne \overline{x_2},$ moreover, since
the homeomorphisms preserve the boundary, $\overline{x_2}\in\partial
D.$ Let $x_1$ and $x_2$ be arbitrary distinct points of the
continuum $A,$ none of which coincide with с $\overline{x_1}.$ By
Lemma~\ref{lem1} we can join points $x_1$ and $\overline{x_1}$ by
the path $\gamma_1:[0, 1]\rightarrow \overline{D},$ and points $x_2$
and $\overline{x_2}$ by the curve $\gamma_2:[0, 1]\rightarrow
\overline{D}$ such that $|\gamma_1|\cap |\gamma_2|=\varnothing,$
$\gamma_i(t)\in D$ for all $t\in (0, 1),$ $i=1,2,$
$\gamma_1(0)=x_1,$ $\gamma_1(1)=\overline{x_1},$ $\gamma_2(0)=x_2$
and $\gamma_2(1)=\overline{x_2}.$ Since $D$ is locally connected on
its boundary there are neighborhoods $U_1$ and $U_2$ of points
$\overline{x_1}$ and $\overline{x_2},$ whose closures do not
intersect, such that $W_i:=D\cap U_i$ is a path-connected set. By
reducing the neighborhood $U_i,$  if necessary, we can assume that
$\overline{U_1}\cap|\gamma_2|=\varnothing=\overline{U_2}\cap|\gamma_1|.$
Without loss of generality, we can assume that $g_m(z_m)\in W_1$ and
$g_m(z^{\,\prime}_m)\in W_2$ for all $m\in {\Bbb N}.$ Let $a_1$ and
$a_2$ be arbitrary points belonging to $|\gamma_1|\cap W_1$ and
$|\gamma_2|\cap W_2.$ Let $t_1, t_2$ be such that
$\gamma_1(t_1)=a_1$ and $\gamma_2(t_2)=a_2.$ We connect the point
$a_1$ with the point $g_m(z_m)$ by the curve $\alpha_m:[t_1,
1]\rightarrow W_1$ such that $\alpha_m(t_1)=a_1$ and
$\alpha_m(1)=g_m(z_m).$ Similarly, we connect a point $a_2$ with the
point $g_m(z^{\,\prime}_m)$ by the curve $\beta_m:[t_2,
1]\rightarrow W_2$ such that $\beta_m(t_2)=a_2$ and
$\beta_m(1)=g_m(z^{\,\prime}_m)$ (see picture~\ref{fig4}).
\begin{figure}[h]
\centerline{\includegraphics[scale=0.6]{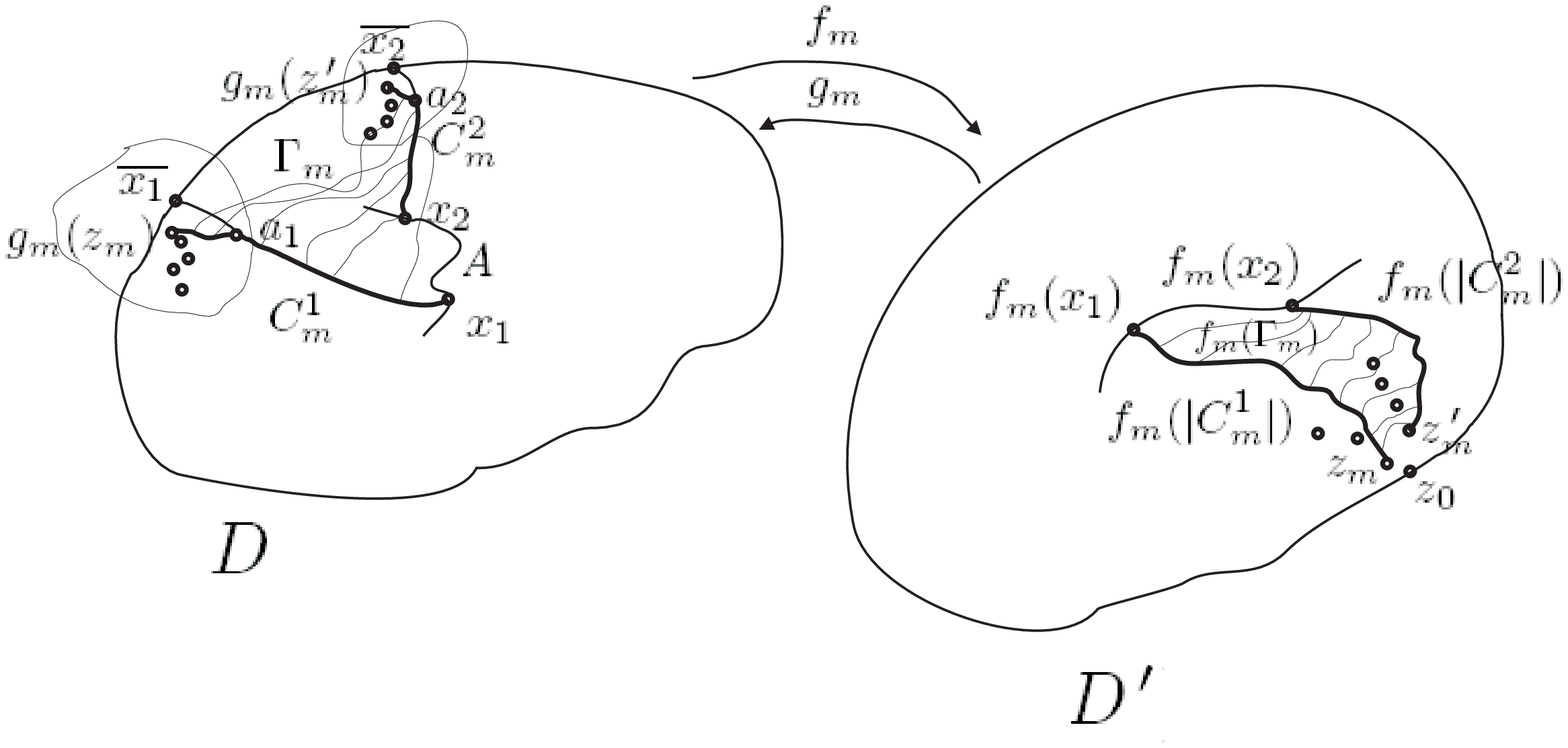}} \caption{To
the proof of the theorem~\ref{th2}}\label{fig4}
\end{figure}
\medskip
We now put
$$C^1_m(t)=\quad\left\{
\begin{array}{rr}
\gamma_1(t), & t\in [0, t_1],\\
\alpha_m(t), & t\in [t_1, 1]\end{array} \right.\,,\qquad
C^2_m(t)=\quad\left\{
\begin{array}{rr}
\gamma_2(t), & t\in [0, t_2],\\
\beta_m(t), & t\in [t_2, 1]\end{array} \right.\,.$$
Let, as usual, $|C^1_m|$ and $|C^2_m|$ are the carriers of the
curves $C^1_m$ and $C^2_m,$  respectively. We note that, by
construction, $|C^1_m|$ and $|C^2_m|$ are two disjoint continua in
$D,$ and ${\rm dist}\,(|C^1_m|, |C^2_m|)>l_0>0$ for all
$m=1,2,\ldots .$ You can take, for example, $$l_0=\min\{{\rm
dist}\,(|\gamma_1|, |\gamma_2|), {\rm dist}\,(|\gamma_1|, U_2), {\rm
dist}\,(|\gamma_2|, U_1), {\rm dist}\,(U_1, U_2)\}\,.$$
Now let $\Gamma_m$ be a family of curves connecting $|C^1_m|$ and
$|C^2_m|$ in $D.$ Then the function
$$\rho(x)= \left\{
\begin{array}{rr}
\frac{1}{l_0}, & x\in D\\
0,  &  x\notin  D
\end{array}
\right. $$
is admissible for the family $\Gamma_m,$ since
$\int\limits_{\gamma}\rho(x)|dx|\geqslant
\frac{l(\gamma)}{l_0}\geqslant 1$ for $\gamma\in\Gamma_m$ (where
$l(\gamma)$ denotes the length of the curve $\gamma$). By
hypothesis, the mappings $f_m,$ $f_m=g_m^{\,-1},$ satisfy
(\ref{eq2*!}) for $Q\in L^1(D),$ so that we obtain:
\begin{equation}\label{eq14A}
M(f_m(\Gamma_m))\leqslant \frac{1}{l_0^n}\int\limits_{D}
Q(x)\,dm(x):=c=c(l_0, Q)<\infty\,.
\end{equation}
On the other hand, by Lemma~\ref{lem3} there is a number
$\delta_1>0$ such that${\rm dist}\,(f_{m}(A), \partial
D^{\,\prime})>\delta_1>0,$ $m=1,2,\ldots \,.$ From this we get that
$${\rm diam}\,f_m(|C^1_m|)\geqslant |z_m-f_m(x_1)| \geqslant
(1/2)\cdot {\rm dist}\,(f_m(A), \partial
D^{\,\prime})>\delta_1/2\,,$$
\begin{equation}\label{eq14}
{\rm diam}\,f_m(|C^2_m|)\geqslant |z^{\,\prime}_m-f_m(x_2)|
\geqslant (1/2)\cdot {\rm dist}\,(f_m(A), \partial
D^{\,\prime})>\delta_1/2
\end{equation}
for some $M_0\in {\Bbb N}$ and for all $m\geqslant M_0.$
%
We choose at the point $z_0\in \partial D^{\,\prime}$ the ball
$U:=B(z_0, r_0),$ where $r_0>0$ and $r_0<\delta_1/4,$ where
$\delta_1$ is the number of the relations in (\ref{eq14}). Notice,
that $f_m(|C^1_m|)\cap U\ne\varnothing\ne f_m(|C^1_m|)\cap
(D^{\,\prime}\setminus U)$ for sufficiently large $m\in{\Bbb N},$
since ${\rm diam\,} f_m(|C^1_m|)\geqslant \delta_1/2$ и $z_m\in
f_m(|C^1_m|),$ $z_m\rightarrow z_0$ for $m\rightarrow\infty.$ In
view of the same considerations $f_m(|C^2_m|)\cap U\ne\varnothing\ne
f_m(|C^2_m|)\cap (D^{\,\prime}\setminus U).$ Since $f_m(|C^1_m|)$
and $f_m(|C^2_m|)$ are continua
\begin{equation}\label{eq8A}
f_m(|C^1_m|)\cap \partial U\ne\varnothing, \quad f_m(|C^2_m|)\cap
\partial U\ne\varnothing\,,
\end{equation}
see~\cite[Theorem~1.I, ch.~5, \S\, 46]{Ku}. For a fixed $P>0,$ let
further $V\subset U$ be a neighborhood of the point $z_0,$
corresponding to the definition of a weakly flat boundary, that is,
such that for any continua $E, F\subset D^{\,\prime}$ with the
condition $E\cap
\partial U\ne\varnothing\ne E\cap \partial V$ и $F\cap \partial
U\ne\varnothing\ne F\cap \partial V$ the inequality holds
\begin{equation}\label{eq9A}
M(\Gamma(E, F, D^{\,\prime}))>P\,.
\end{equation}
We note that for sufficiently large $m\in {\Bbb N}$
\begin{equation}\label{eq10A}
f_m(|C^1_m|)\cap \partial V\ne\varnothing, \quad f_m(|C^2_m|)\cap
\partial V\ne\varnothing\,.\end{equation}
Indeed, $z_m\in f_m(|C^1_m|),$ $z^{\,\prime}_m\in f_m(|C^2_m|),$
where $z_m, z^{\,\prime}_m\rightarrow z_0\in V$ for
$m\rightarrow\infty,$ therefore $f_m(|C^1_m|)\cap V\ne\varnothing\ne
f_m(|C^2_m|)\cap V$ for large $m\in {\Bbb N}.$ In addition, ${\rm
diam}\,V\leqslant {\rm diam}\,U=2r_0<\delta_1/2$ and, since ${\rm
diam }f_m(|C^1_m|)>\delta_1/2$ in view of (\ref{eq14}), then
$f_m(|C^1_m|)\cap (D^{\,\prime}\setminus V)\ne\varnothing.$ Then
$f_m(|C^1_m|)\cap\partial V\ne\varnothing$ (see~\cite[Theorem~1.I,
ch.~5, \S\,~46]{Ku}). Similarly, ${\rm diam}\,V\leqslant {\rm
diam}\,U=2r_0<\delta_1/2$ and since ${\rm
diam}\,f_m(|C^2_m|)>\delta$ in view of (\ref{eq14}), then
$f_m(|C^2_m|)\cap (D^{\,\prime}\setminus V)\ne\varnothing.$ Then
by~\cite[Theorem~1.I, ch.~5, \S\, 46]{Ku} we have:
$f_m(|C^1_m|)\cap\partial V\ne\varnothing.$ Thus, (\ref{eq10A}) is
proved.

\medskip
According to (\ref{eq9A}) and taking into account (\ref{eq8A}) and
(\ref{eq10A}), we get that
%
$$M(f_m(\Gamma_m))=M(\Gamma(f_m(|C^1_m|), f_m(|C^2_m|),
D^{\,\prime}))>P\,,$$
%
which contradicts inequality~(\ref{eq14A}). This contradiction
indicates the incorrectness of the original assumption made in
(\ref{eq12}). The theorem is proved.~$\Box$

\section{Some examples}
We begin with a simple example of mappings on the complex plane.

{\bf Example~1.} As it is known, the linear-fractional automorphisms
of the unit disk ${\Bbb D}\subset{\Bbb C}$ onto itself are given by
the formula $f(z)=e^{i\theta}\frac{z-a}{1-\overline{a}z},$ $z\in
{\Bbb D},$ $a\in{\Bbb C},$ $|a|<1,$ $\theta\in [0, 2\pi).$ The
indicated mappings $f$ are 1-homeomorphisms; all the conditions of
Theorem~\ref{th2} are satisfied, except for the condition ${\rm
diam}\,f(A)\geqslant\delta,$ which, in general, can be violated.

If, for example, $\theta=0$ and $a=1/n,$ $n=1,2,\ldots,$ then
$f_n(z)=\frac{z-1/n}{1-z/n}=\frac{nz-1}{n-z}.$ Let $A=[0, 1/2],$
then $f_n(0)=-1/n\rightarrow 0$ and
$f_n(1/2)=\frac{n-2}{2n-1}\rightarrow 1/2,$ $n\rightarrow\infty.$
Hence we see that the sequence $f_n$ satisfies condition ${\rm
diam}\,f_n(A)\geqslant\delta,$ for example, for $\delta=1/4.$ By
direct calculations we are convinced that
$f_n^{\,-1}(z)=\frac{z+1/n}{1+z/n}$ and, hence, $f_n^{\,-1}$
converge uniformly to $f^{\,-1}(z)\equiv z.$ Thus, the sequence
$f_n^{\,-1}(z)$ is equicontinuous in $\overline{{\Bbb D}}.$

\medskip
If we put
$f^{\,-1}_n(z)=\frac{z-(n-1)/n}{1-z(n-1)/n}=\frac{nz-n+1}{n-nz+1},$
then, as it is easy to see, such a sequence converges locally
uniformly to $-1$ inside ${\Bbb D};$ in the same time,
$f^{\,-1}_n(1)=1.$ Taking this into account, by direct computations,
we conclude that the sequence $f^{\,-1}_n$ is not equicontinuous at
the point 1; in this case $f_n(z)=\frac{z+(n-1)/n}{1+z(n-1)/n}$ and
condition ${\rm diam}\,f_n(A)\geqslant\delta$ for any $\delta>0,$
not depending on $n,$ can not be satisfied in view of
Theorem~\ref{th2}.

\medskip
From what has been said, it follows that {\it in the conditions of
Theorem~\ref{th2}, in general, one can not refuse the additional
requirement that ${\rm diam}\,f(A)\geqslant\delta,$.}

\medskip
{\bf Example~2.} Let $p\geqslant 1$ be so large that the number
$n/p(n-1)$ is less than 1, and let, in addition, $\alpha\in (0,
n/p(n-1))$ be an arbitrary number. We define the sequence of
mappings $f_m: {\Bbb B}^n\rightarrow B(0, 2)$ of the ball ${\Bbb
B}^n$ onto the ball $B(0, 2)$ in the following way:
$$f_m(x)\,=\,\left
\{\begin{array}{rr} \frac{1+|x|^{\alpha}}{|x|}\cdot x\,, & 1/m\leqslant|x|\leqslant 1, \\
\frac{1+(1/m)^{\alpha}}{(1/m)}\cdot x\,, & 0<|x|< 1/m \ .
\end{array}\right.
$$
Notice, that $f_m$ satisfies (\ref{eq2*!}) for
$Q=\left(\frac{1+|x|^{\,\alpha}}{\alpha
|x|^{\,\alpha}}\right)^{n-1}\in L^1({\Bbb B}^n)$ (see~\cite[proof of
Theorem~7.1]{Sev$_3$}) and that $B(0, 2)$ has a weakly flat boundary
(see~\cite[Lemma~4.3]{Vu}). By construction of the mappings $f_m$
fixes an infinite number of points of the unit ball for all
$m\geqslant 2.$

We establish the equicontinuous of mappings $g_m:=f_m^{\,-1}$ in
$\overline{B(0, 2)}$ (for convenience we use the notation $g_m$ also
for continuous extension of the mapping $g_m$ in $\overline{B(0,
2)}$). It is not hard to see that
$$g_m(y):=f^{-1}_m(y)\,=\,\left
\{\begin{array}{rr} \frac{y}{|y|}(|y|-1)^{1/\alpha}\,, & 1+1/m^{\alpha}\leqslant|y|< 2, \\
\frac{(1/m)}{1+(1/m)^{\alpha}}\cdot y\,, & 0<|y|< 1+1/m^{\alpha} \ .
\end{array}\right.
$$
The mappings $g_m$ map $B(0, 2)$ onto ${\Bbb B}^n.$ We fix $y_0\in
\overline{B(0, 2)}.$ The following three situations are possible:

\medskip
1) $|y_0|<1.$ We choose $\delta_0=\delta_0(y_0)$ such that
$\overline{B(y_0, \delta_0)}\subset B(0, 1).$ For the number
$\varepsilon>0$ we put $\delta_1=\delta_1(\varepsilon, y_0):=\min\{
\delta_0, \varepsilon\}.$ In this case, with $y\in \overline{B(y_0,
\delta_1)}$ and all $m=1,2,\ldots$ we have that
$|g_m(y)-g_m(y_0)|=\frac{(1/m)}{1+(1/m)^{\alpha}}|y-y_0|<|y-y_0|<\varepsilon,$
which proves the equicontinuity of the family $g_m$ at the point
$y_0.$

\medskip
2) $|y_0|>1.$ 
By the definition of mappings $g_m$ one can find $m_0=m_0(y_0)\in
{\Bbb N}$ and $\delta_0=\delta_0(y_0)>0$ such that
$g_m(y)=\frac{y}{|y|}(|y|-1)^{1/\alpha}$ for all $\overline{B(y_0,
\delta_0)}\cap \overline{B(0, 2)}$ and all $m\geqslant m_0.$ We take
$\varepsilon>0.$ Putting $g(y)=\frac{y}{|y|}(|y|-1)^{1/\alpha},$ we
note that $|g_m(y)-g_m(y_0)|=|g(y)-g(y_0)|<\varepsilon$ for
$m\geqslant m_0$ and some
$\overline{\delta}=\overline{\delta}(\varepsilon, y_0),$
$\overline{\delta}<\delta_0$, since the mapping
$g(y)=\frac{y}{|y|}(|y|-1)^{1/\alpha}$ since a is continuous in
$\overline{B(0, 2)}.$

\medskip
3) Finally, consider the <<borderline>> case $y_0\in {\Bbb
S}^{n-1}=\partial {\Bbb B}^n.$ Let $\delta_0=\delta_0(y_0)$ be such
that $\overline{B(y_0, \delta_0)}\subset B(0, 2).$ By definition, we
have $g_m(y_0)=\frac{(1/m)}{1+(1/m)^{\alpha}}\cdot y_0,$
$m=1,2,\ldots .$ Notice, that
$$|g_m(y)-g_m(y_0)|\leqslant\max\left\{\left|\frac{(1/m)}{1+(1/m)^{\alpha}}\cdot
y_0-\frac{y}{|y|}(|y|-1)^{1/\alpha}\right|,
\frac{(1/m)}{1+(1/m)^{\alpha}}|y-y_0|\right\}\,.$$
For the number $\varepsilon>0$ we find the number
$m_1=m_1(\varepsilon)>0,$ such that $1/m<\varepsilon/2.$ Put
$\overline{\delta_0}=\overline{\delta_0}(\varepsilon, y_0)=\min\{1,
\varepsilon/2,\delta_0\}.$ Using the triangle inequality, and the
fact that $1/\alpha>1,$ we get:
$\left|\frac{y}{|y|}(|y|-1)^{1/\alpha}-
\frac{(1/m)}{1+(1/m)^{\alpha}}\cdot y_0 \right|\leqslant
(|y|-1)^{1/\alpha}+1/m<\varepsilon/2+\varepsilon/2=\varepsilon$ for
$m>m_1$ and $|y-y_0|<\overline{\delta_0}.$ The last relation for
$1\leqslant m\leqslant m_1$ is also satisfied for $|y-y_0|<\delta_m$
and some $\delta_m=\delta_m(\varepsilon, y_0)>0$ in view of the
continuity of the mappings $g_m.$  Obviously, the same way
$\frac{(1/m)}{1+(1/m)^{\alpha}}|y-y_0|<\varepsilon$ for
$|y-y_0|<\overline{\delta_0}$ and all $m=1,2,\ldots .$ Finally, we
have: $|g_m(y)-g_m(y_0)|<\varepsilon$ for all $m\in {\Bbb N}$ and
$y\in B(y_0, \delta),$ where $\delta:=\{\overline{\delta}_0,
\delta_1,\ldots, \delta_{m_1}\}.$ Equicontinuity of $g_m$ in
$\overline{B(0, 2)}$ is established.

\medskip
It should be noted that the family $\frak G=\{g_m\}_{m=1}^{\infty}$
is equicontinuous in $B(0, 2)$, and the family <<inverse>> to it is
not $\frak F=\{f_m\}_{m=1}^{\infty}$ (indeed,
$|f_m(x_m)-f(0)|=1+1/m\not\rightarrow 0$ for $m\rightarrow\infty,$
where $|x_m|=1/m$).

\medskip
{\it The family $\frak G$ contains an infinite number of mappings
$g_{m_k}:=f^{\,-1}_{m_k},$ $f_{m_k}\in \frak F,$ that do not satisfy
the relation (\ref{eq2*!})}. In fact, otherwise, according to
Theorem~\ref{th1} <<the inverse>>  to $\frak G$ family $\frak F$
would be equicontinuous in ${\Bbb B}^n.$


\medskip
\medskip
{\bf \noindent Evgeny Sevost'yanov, Sergei Skvortsov} \\
Zhytomyr Ivan Franko State University,  \\
40 Bol'shaya Berdichevskaya Str., 10 008  Zhytomyr, UKRAINE \\
Phone: +38 -- (066) -- 959 50 34, \\
Email: esevostyanov2009@gmail.com

\end{document}